\documentclass[10pt]{amsart}
\usepackage{amsfonts,amssymb,amscd,amsmath,enumerate,verbatim,calc}

\setbox0=\hbox{$+$}
\newdimen\plusheight
\plusheight=\ht0
\def\+{\;\lower\plusheight\hbox{$+$}\;}

\setbox0=\hbox{$-$}
\newdimen\minusheight
\minusheight=\ht0
\def\-{\;\lower\minusheight\hbox{$-$}\;}

\setbox0=\hbox{$\cdots$}
\newdimen\cdotsheight
\cdotsheight=\plusheight
\def\cds{\lower\cdotsheight\hbox{$\cdots$}}

\setbox0=\hbox{$+$}
\numberwithin{equation}{section}
 \theoremstyle{plain}
\newtheorem{thm}{Theorem}[section]

\newtheorem{lem}[thm]{Lemma}

\begin{document}
\title[Lucas' Theorem] {A new proof of Lucas' Theorem}
\author{Alexandre Laugier}
\address{Lyc{\'e}e professionnel hotelier La Closerie, 10 rue Pierre Loti - BP 4, 22410 Saint-Quay-Portrieux, France}
\email{laugier.alexandre@orange.fr}
\author{Manjil P.~Saikia}
\thanks{The second author is supported by DST INSPIRE Scholarship 422/2009 from Department of Science and Technology, Government of India.}
\address{Department of Mathematical Sciences, Tezpur University, Napaam, Sonitpur, Assam, Pin-784028, India}
\email{manjil@gonitsora.com}

\maketitle

\begin{abstract}
 We give a new proof of Lucas' Theorem in elementary number theory.
\end{abstract}



\vskip 3mm


\vskip 3mm

\noindent{\footnotesize 2010 Mathematical Reviews Classification
Numbers: 11A07, 11A41, 11A51, 11B50, 11B65, 11B75.}

\section{{Introduction}}

One of the most useful results in elementary number theory is the following result of E.~Lucas

\begin{thm}[\cite{el},E.~Lucas (1878)]\label{lucas}
Let $p$ be a prime and $m$ and $n$ be two integers considered in the following way, $$m=a_kp^k+a_{k-1}p^{k-1}+\ldots+a_1p+a_0,$$ $$n=b_lp^l+b_{l-1}p^{l-1}+\ldots+b_1p+b_0,$$ where all $a_i$ and $b_j$ are non-negative integers less than $p$. Then, $$\binom{m}{n}=\binom{a_kp^k+a_{k-1}p^{k-1}+\ldots+a_1p+a_0}{b_lp^l+b_{l-1}p^{l-1}+\ldots+b_1p+b_0} \equiv \prod_{i=0}^{\max(k,l)}\binom{a_i}{b_i}~(\textup{\textup{mod}}~p).$$
Notice that the theorem is true if $a_i\geq b_i$ for
$i=0,1,2,\ldots,\max(k,l)$
\end{thm}

There has been many different proofs of this result in the years that followed its first publication. We present here an alternate approach using elementary 
number theoretic techniques.

\section{{Proof of Theorem \ref{lucas}}}

First of all, we state and prove a few lemmas

\begin{lem}
If $$
a_0+a_1X+a_2X^2+\ldots+a_nX^n\equiv b_0+b_1X+b_2X^2+\ldots+b_nX^n~(\textup{mod}~p)
$$
then
$$
a_i\equiv b_i~(\textup{mod}~p)\,\,\forall\,i\in [[ 0,n]].
$$
\end{lem}

\begin{proof}
Indeed, if
$a_0+a_1X+a_2X^2+\cdots+a_nX^n\equiv
b_0+b_1X+b_2X^2+\cdots+b_nX^n~(\textup{mod}~p)$, then there exits a polynomial
$k(X)=k_0+k_1X+k_2X^2+\cdots+k_nX^n$ at most of degree $n$ such that
$a_0+a_1X+a_2X^2+\cdots+a_nX^n=
b_0+b_1X+b_2X^2+\cdots+b_nX^n+p(k_0+k_1X+k_2X^2+\cdots+k_nX^n)$. This gives $a_0+a_1X+a_2X^2+\cdots+a_nX^n=
b_0+pk_0+(b_1+pk_1)X+(b_2+pk_2)X^2+\cdots+(b_n+pk_n)X^n$. Hence we get $a_0=b_0+pk_0$, $a_1=b_1+pk_1$,
$a_2=b_2+pk_2$ ,..., $a_n=b_n+pk_n$. Or equivalently
$a_0\equiv b_0~(\textup{mod}~p)$, $a_1=b_1~(\textup{mod}~p)$, $a_2=b_2~(\textup{mod}~p)$ ,$\cdots$,
$a_n=b_n~(\textup{mod}~p)$.

The reciprocal implication is trivial.
\end{proof}

\begin{lem}
If the base $p$ expansion of a positive
integer $n$ is,
$$
n=a_0+a_1p+a_2p^2+\cdots+a_lp^l
$$
then we have
$$
n!=qa_0!(a_1p)!(a_2p^2)!\cdots(a_lp^l)!
$$
with $q$ a natural number.
\end{lem}

\begin{proof}
Since the factorial of a natural number is a natural number,
there exists a rational number q such that
$$
q=\frac{n!}{a_0!(a_1p)!(a_2p^2)!\ldots(a_lp^l)!}.
$$ Let $S$ be the set, $$S=\left\{x_1,x_2,\ldots,x_l\right\}.$$
We consider lists of elements of $S$ where $x_0$ is repeated $a_0$ times, $x_1$
is repeated $a_1p$ times,$\ldots$, $x_l$ is repeated $a_lp^l$ times such that, $0\leq
a_i\leq p-1$ with $i\in [[ 0,l ]]$. In such a list,
there are $l+1$ unlike groups of identical elements. For instance the
selection
$\left(\underbrace{x_0,x_0,\ldots,x_0}_{a_0},\underbrace{x_1,x_1,\ldots,x_1}_{a_1p},
\ldots\ldots,\underbrace{x_l,x_l,\ldots,x_l}_{a_lp^l}\right)$ is such a
list of $n$ elements which contains $l+1$ unlike groups of identical
elements.

The number of these lists is given by $\frac{n!}{a_0!(a_1p)!(a_2p^2)!\ldots(a_lp^l)!}$. It proves that the
rational number $q$ is a natural number. And, since the factorial of a
natural number is non-zero (even if this number is $0$ because
$0!=1$), we deduce that $q$ is a non-zero natural number.
\end{proof}

\begin{thm}\label{qmodp}
Let $n=ap+b=a_0+a_1p+a_2p^2+\ldots+a_lp^l$ such that $0\leq b\leq p-1$ and $0\leq a_i\leq p-1$ with $i\in[[ 0,l]]$.\
Then $q\equiv 1~(\textup{mod}~p)$.
\end{thm}

Before we prove Theorem \ref{qmodp} we shall state and prove the following non-trivial lemmas.

\begin{lem}
The integers $q$ and $p$ are relatively prime.
\end{lem}

\begin{proof}
If $0<q<p$, since $p$ is prime, $q$ and $p$ are relatively
prime.

If $q\geq p$, let us assume that $p$ and $q$ are not relatively
prime. It would imply that there exist an integer $x>0$ and a non-zero
natural number $q'$ such that $q=q'p^x$ with $\textup{gcd}(q',p)=1$. Since
$n!=qa_0!(a_1p)!\ldots(a_{l-1}p^{l-1})!(a_lp^l)!$ we get $n!=q'p^xa_0!(a_1p)!\ldots(a_lp^l)!$. It follows that $n!$ would contain a factor $a_{l+x}p^{l+x}$ such
that $q'=a_{l+x}q''$ with $a_{l+x}\in[[
1,p-1]]$. But, $a_{l+x}p^{l+x}>n$.

Indeed we know that $1+p+\ldots+p^l=\frac{p^{l+1}-1}{p-1}$. So
$p^{l+1}=1+(p-1)(1+p+\ldots+p^{l})$. Then
$p^{l+1}>(p-1)+(p-1)p+\ldots+(p-1)p^l$. Since $0\leq a_i\leq p-1$
with $i\in[[ 0,l]]$, we have $0\leq a_ip^i\leq
(p-1)p^i$ with $i\in[[ 0,l]]$. Therefore,
$p^{l+1}>a_0+a_1p+\ldots+a_lp^l$ so $p^{l+1}>n\,\Rightarrow\,a_{l+x}p^{l+x}>n$.

Since $n!$ doesn't include terms like $a_{l+x}p^{l+x}>n$ with
$a_{l+x}\in[[ 1,p-1]]$, we obtain a contradiction. It
means that the assumption $q=q'p^x$ with $x\in\mathbb{N}^{\star}$ and
$gcd(q',p)=1$ is not correct. So, $q$ is not divisible by a power of
$p$. It results that $q$ and $p$ are relatively prime.
\end{proof}

We know that $n!=(ap+b)!=qa_0!(a_1p)!\ldots(a_lp^l)!$ with
$a=\lfloor\frac{n}{p}\rfloor$, $0\leq a_i\leq p-1$ with
$i=0,1,2,\ldots,p-1$ and $b=a_0$. Let $q_{a,l,1,i}$ with $0\leq i\leq a_1\leq a$ be the natural number
$$
q_{a,l,1,i}=\frac{(ap+b-ip)!}{a_0!((a_1-i)p)!(a_2p^2)!\ldots(a_lp^l)!}.
$$ In particular, we have $q=q_{a,l,1,0}$.

\begin{lem}
$q_{a,l,1,i+1}\equiv q_{a,l,1,i}~(\textup{mod}~p).$
\end{lem}

\begin{proof}
We have ($0\leq i<a_1$)
$$
{ap+b-ip\choose p}=\frac{q_{a,l,1,i}}{q_{a,l,1,i+1}}{(a_1-i)p\choose p}
$$
Or equivalently
$$
q_{a,l,1,i+1}{ap+b-ip\choose p}=q_{a,l,1,i}{(a_1-i)p\choose p}.
$$
Now
$$
{ap+b-ip\choose p}={(a-i)p+b\choose p}\equiv a-i\equiv a_1-i~(\textup{mod}~p),
$$
and
$$
{(a_1-i)p\choose p}\equiv a_1-i~(\textup{mod}~p).
$$
Therefore
$$
q_{a,l,1,i+1}(a_1-i)\equiv q_{a,l,1,i}(a_1-i)~(\textup{mod}~p).
$$
Since $a_1-i$ with $i=0,1,2,\ldots,a_1-1$ and $p$ are relatively
prime, so $
q_{a,l,1,i+1}\equiv q_{a,l,1,i}~(\textup{mod}~p)
$.
\end{proof}

Notice that $q_{a,l,1,a_1}$ corresponds to the case where $a_1=0$, $
(a_0+a_2p^2+\ldots+a_lp^l)!=q_{a,l,1,a_1}a_0!(a_2p^2)!\ldots(a_lp^l)!
$.

\begin{lem}
For $n=ap+b=a_{(k)}p^k+b_{(k)}$ with $0\leq b_{(k)}\leq p^k-1$ and
defining ($1\leq k\leq l$ and $0\leq i\leq a_k$)
$$
q_{a,l,k,i}=\frac{(a_{(k)}p^k+b_{(k)}-ip^k)!}
{a_0!(a_1p)!\ldots((a_k-i)p^k)!\ldots(a_lp^l)!}
$$
with $a_k\geq 1$, (where it is understood that when $a$ and $l$
appears together as the two first labels of one $q_{\cdot\cdot\cdot\cdot}$, it
implies that $a$ is given by
$a=a_{(1)}=(a_1a_2\ldots a_l)_p=a_1+a_2p+\ldots+a_lp^{l-1}$) we have ($0\leq i<a_k$)
$$
{a_{(k)}p^k+b_{(k)}-ip^k\choose p^k}=\frac{q_{a,l,k,i}}{q_{a,l,k,i+1}}
{(a_k-i)p^k\choose p^k}
$$
Additionally, $
q_{a,l,k,i}\equiv q_{a,l,k,i+1}~(\textup{mod}~p)
$.
\end{lem}

In particular for $k=0$, we have $(0\leq i<a_0)$, $$ap+b=\frac{q_{a,l,0,i}}{q_{a,l,0,i+1}}(a_0-i)$$ with $a_0\geq 1$.

We can prove this lemma by the following a similar reasoning as earlier and hence we omit it here.

Notice that $q=q_{a,l,k,0}$. And $q_{a,l,k,a_k}$ corresponds to the case where $a_k=0$. Also
$$
q_{a-i,l,k,0}=q_{a,l,1,i}\equiv q_{a,l,1,i+1}~(\textup{mod}~p),
$$
with $0\leq i<a_1$ and $a_1\geq 1$.
So, since $q_{a,l,k,j}\equiv q_{a,l,k,j+1}~(\textup{mod}~p)$ with $0\leq j<a_k$, we
have $
q_{a-i,l,k,j}\equiv q_{a-i,l,k,0}\equiv q_{a,l,1,i}\equiv q_{a,l,1,0}~(\textup{mod}~p)
$ and $q_{a-i,l,k,0}=q_{a-i,l,l,0}\equiv q_{a-i,l,l,j}\equiv q_{a-i,l,l,a_l}\equiv q_{a-i,l-1,k,0}~(\textup{mod}~p).$

So $q_{a-i,l,k,0}\equiv q_{a-i,l-1,k,0}\equiv \cdots \equiv q_{a-i,l,k,0}\equiv q_{a-i,1,1,0}\equiv q_{a,1,1,i}\equiv q_{a,1,1,0}~(\textup{mod}~p)$. Or $q_{a-i,l,k,0}=q_{a,l,1,i}\equiv q_{a,l,1,0}\equiv q_{a,l,k,0}~(\textup{mod}~p)$.

Finally we have $q=q_{a,l,k,0}\equiv q_{a,1,1,0}~(\textup{mod}~p)$.

\begin{lem}
We have also the congruence $
(a_ip^i)!\equiv a_i!p^{a_i(1+p+\ldots+p^{i-1})}~(\textup{mod}~p).$
\end{lem}

\begin{proof}
We proceed by induction on $i$.

Indeed, we have $(a_1p)!\equiv 1p\cdot 2p\cdots(a_1-1)p\cdot
a_1p~(\textup{mod}~p)$. So $(a_1p)!\equiv a_1!p^{a_1}~(\textup{mod}~p)$.

It follows that $(a_2p^2)!=((a_2p)p)!\equiv (a_2p)!p^{a_2p}\equiv
a_2!p^{a_2}p^{a_2p}~(\textup{mod}~p)$. So $(a_2p^2)!\equiv
a_2!p^{a_2(1+p)}~(\textup{mod}~p)$.

Let us assume that $(a_ip^i)!\equiv
a_i!p^{a_i(1+p+\ldots +p^{i-1})}~(\textup{mod}~p)$.

We have
$(a_{i+1}p^{i+1})!=((a_{i+1}p^{i})p)!\equiv(a_{i+1}p^{i})!p^{a_{i+1}p^i}\equiv
a_{i+1}!p^{a_{i+1}(1+p+\ldots +p^{i-1})}p^{a_{i+1}p^i}~(\textup{mod}~p)$. Thus
$(a_{i+1}p^{i+1})!\equiv a_{i+1}!p^{a_{i+1}(1+p+\ldots +p^{i-1}+p^i)}~(\textup{mod}~p)$.

Hence the result follows.
\end{proof}

We now prove Theorem \ref{qmodp}.

\begin{proof}
If $a_i=0$ for $i\in[[ 1,l]]$, then $n=a_0$ and we
have $n!=qa_0!$ with $q=1$. So, if $a_i=0$ for $i\in[[
1,l]]$, $q\equiv 1~(\textup{mod}~p)$.

Let consider the case where $a_i=0$ for all $i>1$, then $n=a_0+a_1p$ and we have
$n!=qa_0!(a_1p)!$. Then
$$
qa_0!=\frac{n!}{(a_1p)!}=(a_0+a_1p)(a_0+a_1p-1)\ldots(a_1p+1)
={\displaystyle\prod^{a_0-1}_{r=0}(a_1p+a_0-r)}.
$$
Since $0<a_0-r\leq a_0$ for $r\in[[ 0,a_0-1]]$, we
obtain
$$
qa_0!\equiv {\displaystyle\prod^{a_0-1}_{r=0}(a_0-r)}\equiv{\displaystyle\prod^{a_0}_{r=1}r}\equiv a_0!~(\textup{mod}~p).
$$

Since $a_0!$ and $p$ are relatively prime, we have $
q\equiv 1~(\textup{mod}~p)
$.

Since $q=q_{a,l,k,0} \equiv q_{a,l,1,0}~(\textup{mod}~p)$, we conclude that $q\equiv 1~(\textup{mod}~p)$
whatever $n$ is.
\end{proof}

We come back to the proof of Theorem \ref{lucas}.

\begin{proof}
Let $m,n$ be two positive integers whose base $p$ expansion with $p$
a prime, are
$$
m=a_0+a_1p+\ldots+a_kp^k,
$$
and
$$
n=b_0+b_1p+\ldots+b_lp^l,
$$
such that $m\geq n$. We assume that $a_i\geq b_i$ with
$i=0,1,2,\ldots,\max(k,l)$. We denote $a=\lfloor\frac{m}{p}\rfloor$ and $b=\lfloor\frac{n}{p}\rfloor$. Since $a_i\geq b_i$ with
$i=0,1,2,\ldots,\max(k,l)$, we have $a\geq b$. We define
$$
a_{\max(k,l)}=\left\{\begin{array}{ccc}
0 & {\rm if} & k<l\\
a_k & {\rm if} & k\geq l
\end{array}
\right.
$$
and
$$
b_{\max(k,l)}=\left\{\begin{array}{ccc}
b_l & {\rm if} & k\leq l\\
0 & {\rm if} & k>l
\end{array}
\right.
$$
In particular if $\max(k,l)=k$, $b_i=0$ for $i>l$ and if
$\max(k,l)=l$, $a_i=0$ for $i>k$. Since $m\geq n$, $\max(k,l)=k$. So,
we have $a_{\max(k,l)}=a_k$ and $b_{\max(k,l)}=b_k$ with $b_k=0$ when $l<k$.

Using these
$$
m!=q_{a,k,1,0}a_0!(a_1p)!\ldots(a_kp^k)!,
$$
$$
n!=q_{b,l,1,0}b_0!(b_1p)!\ldots(b_lp^l)!,
$$
and
$$
(m-n)!=q_{a-b,k,1,0}(a_0-b_0)!((a_1-b_1)p)!\ldots((a_k-b_k)p^k)!.
$$
We have
$$
{m\choose n}=\frac{q_{a,k,1,0}}{q_{b,l,1,0}q_{a-b,k,1,0}}{a_0\choose b_0}
{a_1p\choose b_1p}\ldots{a_{\max(k,l)}p^{\max(k,l)}\choose b_{\max(k,l)}p^{\max(k,l)}}.
$$
Rearranging
$$
q_{b,l,1,0}q_{a-b,k,1,0}{m\choose n}=q_{a,k,1,0}{a_0\choose b_0}
{a_1p\choose b_1p}\ldots{a_{\max(k,l)}p^{\max(k,l)}\choose b_{\max(k,l)}p^{\max(k,l)}}.
$$
Since $q_{a,k,1,0}\equiv q_{b,l,1,0}\equiv q_{a-b,k,1,0}\equiv 1~(\textup{mod}~p)$,
we get
$$
{m\choose n}\equiv {a_0\choose b_0}
{a_1p\choose b_1p}\ldots{a_{\max(k,l)}p^{\max(k,l)}\choose b_{\max(k,l)}p^{\max(k,l)}}~(\textup{mod}~p).
$$
Notice that if for some $i$, $a_i=0$, then $b_i=0$ since we assume
that $a_i\geq b_i$ and $b_i\geq 0$. In such a case ${a_ip^i\choose b_ip^i}={a_i\choose b_i}=1$.

We assume that $a_i\geq 1$. For $k\in[[ 1,a_ip^i-1]]$ and for $i\in[[
0,\max(k,l)]]$ with $1\leq a_i\leq p^i-1$, ${a_ip^i\choose k}\equiv 0~(\textup{mod}~p)$

Therefore for $a_i=1$ we have
$$
(1+x)^{p^i}={\displaystyle\sum^{p^i}_{k=0}{p^i\choose k}x^k}\equiv 1+x^{p^i}~(\textup{mod}~p),
$$
and for any $a_i\in[[ 1,p^i-1]]$
$$
(1+x)^{a_ip^i}=((1+x)^{p^i})^{a_i}\equiv (1+x^{p^i})^{a_i}~(\textup{mod}~p).
$$

Now comparing
$$
(1+x)^{a_ip^i}
={\displaystyle\sum^{a_ip^i}_{k=0}{a_ip^i\choose k}x^k},
$$
and
$$
(1+x^{p^i})^{a_i}={\displaystyle\sum^{a_i}_{l=0}{a_i\choose l}x^{lp^i}}
$$
we get by taking $k=b_ip^i$ and $l=b_i$,
$$
{a_ip^i\choose b_ip^i}\equiv {a_i\choose b_i}~(\textup{mod}~p).
$$
Finally we have
$$
{m\choose n}\equiv {a_0\choose b_0}
{a_1\choose b_1}\ldots{a_{\max(k,l)}\choose b_{\max(k,l)}}~(\textup{mod}~p).
$$
\end{proof}

\end{document}